\documentclass[12pt]{amsart}

\usepackage{latexsym,amssymb,amsmath,bm,enumitem,verbatim,todonotes,stmaryrd}
\usepackage[hidelinks]{hyperref}

\makeatletter
\@namedef{subjclassname@2020}{%
  \textup{2020} Mathematics Subject Classification}
\makeatother

\title{The ordering principle and higher dependent choice}

\usepackage{bbm}

\DeclareMathOperator{\Ord}{Ord}

\DeclareMathOperator{\range}{range}

\DeclareMathOperator{\dom}{dom}

\DeclareMathOperator{\id}{id}

\DeclareMathOperator{\forces}{\Vdash}

\DeclareMathOperator{\GCH}{GCH}

\DeclareMathOperator{\ZFC}{ZFC}
\DeclareMathOperator{\ZF}{ZF}

\DeclareMathOperator{\PP}{\mathbb P}

\newcommand{\LL}{\mathrm{L}}

\newcommand{\DC}{\mathrm{DC}}
\newcommand{\AC}{\mathrm{AC}}
\DeclareMathOperator{\fix}{\mathrm{fix}}
\newcommand{\G}{\mathcal G}

\newcommand{\SSS}{\mathcal S}

\newcommand{\add}{\mathrm{add}}

\newcommand{\HS}{\mathrm{HS}}
\newcommand{\sym}{\mathrm{sym}}

\newcommand{\sG}{\mathcal{G}}
\newcommand{\sS}{\mathcal{S}}
\newcommand{\sF}{\mathcal{F}}

\newcommand{\OP}{\mathrm{OP}}

\DeclareMathOperator{\ran}{ran}

\author{Peter Holy}
\address{Institut f\"ur diskrete Mathematik und Geometrie\\
TU Wien\\
Wiedner Hauptstrasse 8-10/104\\
1040 Vienna\\
Austria}
\email{peter.holy@tuwien.ac.at}

\author{Jonathan Schilhan}
\address{University of Vienna\\
Institute of Mathematics\\
Kurt Gödel Research Center\\
Kolingasse 14-16\\
1090 Vienna\\
Austria}
\email{jonathan.schilhan@univie.ac.at}


\thanks{This research was funded in whole or in part by the Austrian Science Fund (FWF) [10.55776/ESP5711024]. For open access purposes, the authors have applied a CC BY public copyright license to any author-accepted manuscript version arising from this submission.}

\begin{document}

\begin{abstract}
  We provide, for any regular uncountable cardinal $\kappa$, a new argument for Pincus' result on the consistency of $\ZF$ with the higher dependent choice principle
   $\DC_{<\kappa}$ and the ordering principle in the presence of a failure of the axiom of choice. We also generalise his methods and obtain these consistency results in a larger class of models.
\end{abstract}

\maketitle

\newcommand{\chkfront}{\scalebox{1.6}[1.0]{$\vee$}}

\newtheorem{fact}{Fact}
\newtheorem{lemma}[fact]{Lemma}
\newtheorem{theorem}[fact]{Theorem}
\newtheorem{corollary}[fact]{Corollary}
\newtheorem{claim}[fact]{Claim}
\newtheorem{subclaim}[fact]{Subclaim}
\newtheorem{conjecture}[fact]{Conjecture}
\newtheorem{observation}[fact]{Observation}
\newtheorem{proposition}[fact]{Proposition}
\newtheorem{counterexample}[fact]{Counterexample}
\theoremstyle{definition}
\newtheorem{question}[fact]{Question}
\newtheorem{remark}[fact]{Remark}
\newtheorem{definition}[fact]{Definition}

\section{Introduction}

The ordering principle, $\OP$, is the statement that every set can be linearly ordered.
The axiom of choice, $\AC$, in one of its equivalent forms, states that every set can be wellordered, and thus clearly implies $\OP$. If $\delta$ is an infinite cardinal, the principle $\DC_{\delta}$ of higher dependent choice can be stated as follows: whenever~$T$ is a tree without terminal nodes that is closed under increasing sequences of length less than $\delta$, then it contains an increasing sequence of length~$\delta$. Note that by an easy argument (see \cite[Section 8]{jech}), these principles become stronger as $\delta$ increases. 
The principle of dependent choice $\DC$, that is the statement that whenever $R$ is a relation on a set $X$ with the property that $\forall x\in X\,\exists y\in X\ x\,R\,y$ there exists a sequence $\langle x_i\mid i<\omega\rangle$ of elements of $X$ such that $\forall i<\omega\ x_i\,R\,x_{i+1}$, is easily seen to be equivalent to $\DC_\omega$. Finally, for an uncountable cardinal $\lambda$, $\DC_{<\lambda}$ denotes the statement that $\DC_\delta$ holds whenever $\delta<\lambda$ is a cardinal.

\medskip

In his \cite{pincus}, Pincus provided two arguments for the consistency of $\ZF+\OP+\DC+\lnot\AC$ (in fact, $\lnot\DC_{\omega_1}$). His first argument builds on the basic Cohen model (adding countably many Cohen subsets of~$\omega$ and then passing to a symmetric submodel where $\AC$, but also $\DC$ fails), and then adding certain maps on top of that, in order to resurrect $\DC$.
We provided a modern presentation of this result in our~\cite{hs}. Pincus then goes on to argue for the (stronger) consistency of $\ZF+\OP+\DC_{<\lambda}+\lnot\AC$ (in fact, $\lnot\DC_\lambda$) for an arbitrary regular and uncountable cardinal~$\lambda$ (while preserving cardinals at least up to and including $\lambda$). However, there is consensus among the current experts in the field that Pincus' arguments are not understood, and that verifying their correctness is practically impossible.
The main reason for this is Pincus' peculiar approach to the forcing technique, which lacks proper presentation in the literature, and was outdated already when \cite{pincus} was published.
Isolating a notion of \emph{hereditary almost disjointness} from \cite{pincus}, we came up with a related notion of hereditarily almost disjoint towers, and, making use of those as our basic building blocks, eventually with a modern proof of Pincus' claimed result.\footnote{This also yields a different (and in fact, probably somewhat easier than the one provided in~\cite{hs}) argument for the consistency of $\ZF+\OP+\DC+\lnot\AC$.}

Throughout this paper, let $\kappa$ be a fixed regular and uncountable cardinal, and let~$\lambda$ be a fixed regular cardinal such that either $\kappa=\lambda$ or $\kappa<\kappa^+<\lambda$.
Over a suitable ground model (for example, G\"odel's constructible universe), we will obtain the above consistency result starting with $\add(\kappa,\lambda)$, the standard forcing notion to add $\lambda$-many Cohen subsets of $\kappa$, and then continuing in $\lambda$-many steps, where at each stage $0<\alpha<\lambda$, we add $\lambda$-many maps from cardinals less than $\lambda$ to the set of things that we have added so far, in a careful way. We finally obtain our desired model by passing to a suitable symmetric submodel of the above-described forcing extension of our universe. While the very basic construction may seem somewhat similar to the one that we presented in \cite{hs} at first glance, both the construction and the arguments here are in fact very much different. Note that the choice of $\kappa$ in the construction is not relevant to the particular consistency result we obtain, but varying $\kappa$ gives rise to different models witnessing it.\footnote{Also note that in the above, we exclude the case when $\lambda=\kappa^+$, as our arguments will not work in this case.}

\section{Hereditarily almost disjoint towers}\label{section:towers}

A key ingredient of our constructions will be hereditarily almost disjoint (or HAD) towers, which we will introduce in the following.\footnote{The actual conditions that we will use for our forcing notion, that we will define in the next section of this paper, will contain further information (or in order to be somewhat more specific already, this part of our conditions will then work on adding $\lambda$-many Cohen subsets of $\kappa$), for which we will leave space at level $0$ of our towers below.}

\begin{definition}\label{definition:tower}
  We say that $p$ is a \emph{$\lambda$-tower} if:
  \begin{itemize}
    \item $p$ is a function with domain $\dom(p)\subseteq(\lambda\setminus\{0\})\times\lambda$ and $|\dom(p)|<\lambda$,
    \item If $(\alpha,\beta)\in\dom(p)$, then for some nonzero cardinal $\delta<\lambda$,\[p(\alpha,\beta)\colon\delta\to\alpha\times\lambda\] is an injection.
    \item If $(\alpha,\beta_0)$ and $(\alpha,\beta_1)$ are both in $\dom p$, then $p(\alpha,\beta_0)\ne p(\alpha,\beta_1)$.
  \end{itemize}
  Given $\lambda$-towers $p$ and $q$, we say that $q$ \emph{extends} $p$, and write $q \leq p$, if $q\supseteq p$.
 \end{definition}

We will write $p_{\alpha,\beta}$ or $p_{(\alpha,\beta)}$ rather than $p(\alpha,\beta)$. Since $\lambda$ will be fixed throughout our paper, we will simply write \emph{tower} rather than $\lambda$-tower.  

\begin{definition}
  Let $p$ be a tower. We define the \emph{target} of $p$ to be \[t(p)=\dom p\cup\bigcup_{\gamma\in\dom(p)}\range p_\gamma.\]
  We say that $p$ is \emph{complete} if $t(p)\setminus(\{0\}\times\lambda)=\dom(p)$.
\end{definition}

Note that by the regularity of $\lambda$, $|t(p)|<\lambda$.
Given towers $p$ and $q$, we say that they are \emph{compatible} if there is a tower~$r$ such that $r\le p,q$. Note that in this case, $p\cup q$ is their (unique) greatest lower bound in the ordering of towers. Similarly, if $\{p_i\mid i\in I\}$ is a family of towers that has a common lower bound with respect to $\le$, $\bigcup_{i\in I}p_i$ is their greatest lower bound, which is again a tower.
Note that whenever a union of complete towers is a tower, then it is complete.

\begin{definition}
  Given a complete tower $p$, and a set $e\subseteq\lambda\times\lambda$, we define the \emph{target} $t(p,e)\subseteq\lambda\times\lambda$ of $p$ on $e$, by inductively defining a sequence $\langle t^n(p,e)\mid n<\omega\rangle$, with each $t^n(p,e)\subseteq t(p)$, and then taking $t(p,e)=\bigcup_{n<\omega}t^n(p,e)$, as follows:
\begin{itemize}
  \item $t^0(p,e)=e\cap t(p)$.
  \item Given $t^n(p,e)$, let \[t^{n+1}(p,e)=t^n(p,e)\ \cup\bigcup\{\range p_\gamma\mid\gamma\in\,t^n(p,e)\setminus(\{0\}\times\lambda)\}.\]
\end{itemize}
Note that if $\alpha<\lambda$ is such that $e\subseteq\alpha\times\lambda$, then also $t(p,e)\subseteq\alpha\times\lambda$. Note also that $t(p,\lambda\times\lambda)=t(p,t(p))=t(p)$.
\end{definition}

\begin{definition}(HAD towers) Let $p$ be a complete tower. If $d\subseteq t(p)$, we say that~$d$ is \emph{finitely generated} (in~$p$) if there is a finite set $e\subseteq d$ such that $d=t(p,e)$. We also say that $d$ is (finitely) \emph{generated} by $e$ (in~$p$) in this case.
  We say that $p$ is \emph{hereditarily almost disjoint}, or \emph{HAD}, if whenever $\gamma_0,\gamma_1\in t(p)$, then $t(p,\{\gamma_0\})\cap t(p,\{\gamma_1\})$ is finitely generated (in $p$).
\end{definition}

By the finitary nature of the HAD property, any $\le$-decreasing ${<}\lambda$-sequence of HAD towers has a HAD tower as its greatest lower bound.
Adding elements to the target of a HAD tower is essentially trivial:

\begin{lemma}\label{lemma:addtotarget}
  If $p$ is a HAD tower, and $\alpha,\beta<\lambda$ with $(\alpha,\beta)\not\in t(p)$, then there is a HAD tower $q\le p$ such that
  \begin{itemize}
      \item  $(\alpha,\beta)\in t(q)$ and
      \item $t(q)$ is the disjoint union $t(q)=t(p)\cup t(q,\{(\alpha,\beta)\})$.
  \end{itemize}
\end{lemma}
\begin{proof}
  If $\alpha=0$, pick $\bar\beta$ such that $(1,\bar\beta)\not\in\dom(p)$. Let $q_{1,\bar\beta}$ be the function with domain $1$ that maps $0$ to $(0,\beta)$, and let $q_\gamma=p_\gamma$ for $\gamma\in\dom(p)$. If $\alpha>0$, pick $\bar\beta<\lambda$ such that $(0,\bar\beta)\not\in t(p)$, let $q_{\alpha,\beta}$ be the function with domain~$1$ that maps $0$ to $(0,\bar\beta)$, and let $q_\gamma=p_\gamma$ for $\gamma\in\dom(p)$. Note that in both cases, since $q_{\alpha,\beta}\ne p_{\alpha,\beta'}$ whenever $(\alpha,\beta')\in\dom(p)$, $q$ is a complete tower, and it obviously has the two properties listed in the statement of the lemma. The HAD property of $q$ trivially follows from the HAD property of $p$ together with the second of these properties.
\end{proof}

An easy to verify, yet crucial property of HAD towers is that they can be extended so that the range of a single element covers the target of the original tower.

\begin{lemma}\label{lemma:singletoncover}
  If $p$ is a HAD tower, then there is a HAD tower $q\le p$ and an ordinal $\alpha^*<\lambda$ such that:
  \begin{itemize}
      \item $t(q)=t(q,\{(\alpha^*,0)\})$.
      \item $t(p)=\ran(q_{\alpha^*,0})$.
  \end{itemize}
\end{lemma}
\begin{proof}
  Pick $\alpha^*<\lambda$ such that $\dom(p)\subseteq\alpha^*\times\lambda$. Let $t(p)$ be enumerated by $\langle t_\epsilon\mid \epsilon<\delta\rangle$ for a cardinal $\delta<\lambda$.
  Extend $p$ to a complete tower $q\le p$ by setting $q_{\alpha^*,0}=\langle t_\epsilon\mid \epsilon<\delta\rangle$, and letting $q_\gamma=p_\gamma$ otherwise. We need to check that $q$ is a HAD tower. Note that if $\gamma\in t(q)$, then $t(q,\{\gamma\})\cap t(q,\{(\alpha^*,0)\})=t(q,\{\gamma\})$, which is finitely generated (by $\{\gamma\}$). If $\gamma_0,\gamma_1\in t(q)$ are both different to $(\alpha^*, 0)$, i.e., elements of $t(p)$, then $$t(q,\{\gamma_0\})\cap t(q,\{\gamma_1\})=t(p,\{\gamma_0\})\cap t(p,\{\gamma_1\}),$$ which is finitely generated in the HAD tower $p$, and thus also in $q$.
\end{proof}

\begin{lemma}\label{lemma:finitehad}
    Let $p$ be a HAD tower, let $n\in\omega$, and let $\gamma_0, \dots, \gamma_n\in t(p)$. Then, $\bigcap_{i \leq n} t(p, \{\gamma_i\})$ is finitely generated (in $p$).
\end{lemma}

\begin{proof}
    First note that for any $e \subseteq t(p)$, $t(p, e) = \bigcup_{\gamma \in e} t(p, \{\gamma\})$. We verify the lemma by induction on $n$. The case $n=0$ is trivial. Suppose inductively that the lemma is true for a particular value $n \geq 0$, and let $\gamma_0, \ldots, \gamma_n, \gamma_{n+1} \in t(p)$. Then, \begin{align*}
        \bigcap_{i \leq n +1} t(p, \{ \gamma_i \}) &= \left( \bigcap_{i \leq n} t(p, \{ \gamma_i \})\right) \cap t(p, \{ \gamma_{n+1}\}) \\ &= t(p, e) \cap t(p, \{ \gamma_{n+1} \})\\ &= \bigcup_{\gamma \in e} \big( t(p, \{ \gamma \}) \cap t(p, \{ \gamma_{n+1} \})\big) \\ &= \bigcup_{\gamma \in e} t(p, e_\gamma) = t(p, \bigcup_{\gamma \in e} e_\gamma),
    \end{align*}
for appropriate finite $e \subseteq t(p)$ and $e_\gamma \subseteq t(p)$ for $\gamma \in e$, using the HAD property and our inductive hypothesis.
\end{proof}

\section{Our forcing notion}\label{section:forcing}

The forcing notion that we use will be the product $P_0\times P_1$, where $P_0=\add(\kappa,\lambda)$ and $P_1$ is the set of all HAD towers, ordered by extension as in Definition \ref{definition:tower}. Let us agree that whenever $I\subseteq\Ord$, we think of conditions $q$ in $\add(\kappa,I)$, the standard forcing notion to add a Cohen subset of $\kappa$ for every $i\in I$, as sequences $\langle q_\alpha\mid\alpha\in J\rangle$ with a domain $J$ that is a ${<}\kappa$-size subset of $I$, and with sequents being functions from some ordinal less than~$\kappa$ to~$2$. These conditions are ordered by componentwise reverse inclusion, as usual. For the sake of simplicity of notation, conditions $p=(p_0,\bar p)\in P=P_0\times P_1$ will also be written as \[p=\langle p_{\alpha,\beta}\mid(\alpha=0\,\land\,\beta\in\dom p_0)\,\lor\,(\alpha>0\,\land\,(\alpha,\beta)\in\dom(\bar p)\rangle.\] We let $\dom p=(\{0\}\times\dom p_0)\cup\dom\bar p$, and we think of $p$ as a function with domain $\dom p$.
We let $t(p)=(\{0\}\times\dom p_0)\cup t(\bar p)$, and also $t(p,e)=(e\cap(\{0\}\times\dom p_0))\cup t(\bar p,e)$ whenever $e\subseteq\lambda\times\lambda$.
If $\alpha<\lambda$, we also let $p_\alpha=\langle p_{\alpha,\beta}\mid\beta<\lambda\,\land\,(\alpha,\beta)\in\dom(p)\rangle$ and we let $\dom p_\alpha=\{\beta\mid(\alpha,\beta)\in\dom p\}$.

Assume the $\GCH$, and that there is a global wellorder (say for example that we start in $\LL$).\footnote{It is easy to see that the $\GCH$ could be replaced by somewhat weaker assumptions here; we will leave the details of figuring out what exactly is needed to the interested reader.}  $P_0=\add(\kappa,\lambda)$ is ${<}\kappa$-closed and $\kappa^+$-cc. Since HAD towers are closed under decreasing ${<}\lambda$-unions, $P_1$ is ${<}\lambda$-closed. Using the $\GCH$, $P$ is also of size $\lambda$, so forcing with $P$ preserves all cardinals.\footnote{It would be enough for a meaningful result if it preserved all cardinals ${\le}\lambda$.}

\medskip

For any $\beta<\lambda$, let $\dot g_{0,\beta}$ be the canonical $P_0=\add(\kappa,\lambda)$-name, which we can also think of as a $P$-name, for the $\beta^\textrm{th}$ Cohen subset of $\kappa$ added. We now proceed to define further objects inductively. Given $0<\alpha<\lambda$, assume that we have defined $\dot g_{\bar\alpha,\beta}$ whenever $\bar\alpha<\alpha$ and $\beta<\kappa$. We also allow for the notation $\dot g_{(\bar\alpha,\beta)}$ rather than $\dot g_{\bar\alpha,\beta}$. For every $\beta<\kappa$, let $\dot g_{\alpha,\beta}$ denote the canonical $P$-name for the function with domain $\dom p_{\alpha,\beta}$ mapping any given $\epsilon\in\dom p_{\alpha,\beta}$ to $\dot g_{p_{\alpha,\beta}(\epsilon)}$ whenever $p$ is a HAD tower in the generic filter with $(\alpha,\beta)\in t(p)$. To be precise, $$\dot g_{\alpha, \beta} := \left\{ \left(p, (\check\epsilon, \dot g_{p_{\alpha,\beta}(\epsilon)})^\bullet\right) \mid p \in P, (\alpha,\beta) \in t(p),\epsilon\in\dom p_{\alpha,\beta} \right\}.\footnote{Given a finite tuple $(\dot x_0,\ldots,\dot x_n)$ of $P$-names, $(\dot x_0,\ldots,\dot x_n)^\bullet$ denotes the canonical $P$-name for the tuple consisting of the evaluations of the $\dot x_i$. Likewise, for a set $X$ of $P$-names, $X^\bullet$ denotes the canonical $P$-name for the set containing exactly the evaluations of the elements of $X$. For any set $I$, $\langle \dot x_i\mid i\in\check I\rangle^\bullet$ denotes the canonical $P$-name for the $I$-sequence of evaluations of the $\dot x_i$.}$$
 For every $\alpha<\lambda$, let $\dot A_\alpha=\{\dot g_{\alpha,\beta}\mid\beta<\lambda\}^\bullet$, and for $\alpha\le\lambda$, let $\dot A_{<\alpha}=\bigcup_{\bar\alpha<\alpha}\dot A_{\bar\alpha}$. Let $\dot A=\dot A_{<\lambda}$. If $G$ is $P$-generic, $\alpha,\beta<\lambda$, and we are in a context where $G$ is the only $P$-generic that we currently make use of, we let $g_{\alpha,\beta}=\dot g_{\alpha,\beta}^G$, $A_\alpha=\dot A_\alpha^G$ etc.
 Let $\dot G$ be the canonical $P$-name for the $P$-generic filter.
 
\section{Our symmetric system}\label{section:symmetricsystem}

We next define a symmetric system $\SSS=\langle P,\sG,\sF\rangle$ using the notion of forcing $P$ that we have already defined above.

\begin{definition}
  Let $\sG$ be the set of sequences $\pi=\langle\pi_\alpha\mid\alpha<\lambda\rangle$ of permutations of $\lambda$, with each sequent moving only less than $\lambda$-many ordinals, and with only less than $\lambda$-many nontrivial sequents, which form a group using componentwise composition. Given such $\pi$, we let $\pi$ act on $\lambda\times\lambda$, letting, for $(\alpha,\beta)\in\lambda\times\lambda$, $\pi((\alpha,\beta))=(\alpha,\pi_\alpha(\beta))$. If $\delta<\lambda$ is a cardinal and $f\colon\delta\to\lambda\times\lambda$, we let $\pi(f)$ be the function with domain $\delta$ such that $\pi(f)(\epsilon)=\pi(f(\epsilon))$ for every $\epsilon<\delta$. We let $\pi\in\sG$ act on a condition $p\in P$ as follows:
  \begin{itemize}
      \item $\dom \pi(p)_\alpha=\pi_{\alpha}[\dom p_\alpha]$ for every $\alpha<\lambda$.
      \item $\pi(p)_{0,\pi_0(\beta)}=p_{0,\beta}$ whenever $\beta\in\dom p_0$.
      \item $\pi(p)_{\alpha,\pi_\alpha(\beta)}=\pi(p_{\alpha,\beta})$ whenever $\alpha>0$ and $\beta\in\dom p_\alpha$.
  \end{itemize}
\end{definition}

Note that for every $e\subseteq t(p)$, $t(\pi(p),\pi[e])=\pi[t(p,e)]$. This implies that the HAD property is preserved from $p$ to $\pi(p)$, that is $\pi(p)\in P$.

\medskip

We use finite support to define our filter $\sF$ on the set of subgroups of $\G$, that is,~$\sF$ is generated by the subgroups $\fix(e)=\{\pi\in\sG\mid\pi\restriction e=\id\}\le\G$ for $e\subseteq\lambda\times\lambda$ finite. Note that $\pi \fix(e) \pi^{-1} = \fix(\pi[e])$, so $\sF$ is indeed a normal filter. The \emph{symmetry group} of a $P$-name $\dot x$ is $\sym(\dot x)=\{\pi\in\sG\mid\pi(\dot x)=\dot x\}$, and if $\fix(e)\le\sym(\dot x)$, we also say that $e$ is a \emph{support} of $\dot x$.

\medskip

Note that for $\alpha,\beta<\lambda$, $\pi(\dot g_{\alpha,\beta})=\dot g_{\pi(\alpha,\beta)}=\dot g_{\alpha,\pi_\alpha(\beta)}$. In particular, each $\dot g_{\alpha,\beta}$ is symmetric, with symmetry group $\fix(\{(\alpha,\beta)\})$. Moreover, each $\dot A_\alpha$ is symmetric with symmetry group $\sG$, as is each $\dot A_{<\alpha}$, and also $\langle \dot A_\alpha\mid\alpha<\kappa\rangle^\bullet$.

\medskip

We will later use the following standard fact, which says that we can uniformly find names for definable objects. We include the short proof for the convenience of our readers.

\begin{fact}\label{fact:namebydef}
Let $\varphi(u,v_0, \dots, v_n)$ be a formula in the language of set theory. Then, there is a definable class function $F$ so that for any $\sS$-names $\dot x_0, \dots, \dot x_n$ and $p\in P$ with $$p \Vdash_{\sS} \exists! y \varphi(y, \dot x_0, \dots, \dot x_n),$$ $\dot y=F(p, \dot x_0, \dots, \dot x_n)$ is an $\sS$-name with $\bigcap_{i \leq n} \sym(\dot x_i) \leq \sym(\dot y)$ so that $$p \Vdash_{\sS} \varphi(\dot y, \dot x_0, \dots, \dot x_n).$$
\end{fact} 

\begin{proof}
Let $\gamma$ be the least ordinal such that \[p \Vdash_{\sS} \exists y \in \HS_\gamma^\bullet\ \varphi(y, \dot x_0, \dots, \dot x_n).\] Let $F(p, \dot x_0, \dots, \dot x_n)=\dot y$ be the set of all pairs $(q, \dot z) \in \PP \times \HS_\gamma$ so that $$q \Vdash \forall y  (\varphi(y, \dot x_0, \dots, \dot x_n) \rightarrow \dot z \in y) \}.$$
\end{proof}

\section{The failure of AC}\label{section:nonac}

We first verify a fairly general lemma.

\begin{lemma}[Restriction Lemma]\label{lem:restriction}
    Let $\varphi$ be a formula in the language of set theory and let $\dot x$ be an $\sS$-name with support $e \in [\lambda \times \lambda]^{<\omega}$. Whenever $p \Vdash_{\sS} \varphi(\dot x)$, already the restriction $p\restriction t(p,e)$ of $p$ to  $t(p, e)$, defined in the obvious way, forces $\varphi(\dot x)$.
\end{lemma}

\begin{proof}
    Let $q\le p\restriction t(p,e)$. Pick a permutation $\pi=\langle\pi_\gamma\mid\gamma<\lambda\rangle\in\sG$ such that $\pi$ fixes $t(p,e)=t(q,e)$ pointwise, and which swaps $t(q)\setminus t(q,e)$ with a set that is disjoint from $t(q)$. Such $\pi$ can easily be found. We will show that $p\parallel\pi(q)$, hence $\pi(q)$ cannot force $\lnot\varphi(\dot x)$, hence, applying $\pi^{-1}$ to this forcing statement, also $q$ cannot force $\lnot\varphi(\dot x)$, yielding that $p\restriction t(p,e)\forces\varphi(\dot x)$. We will in fact verify the stronger statement that $q\parallel\pi(q)$.\footnote{This is a stronger statement since $q$ is arbitrary with $q\le p\restriction t(p,e)$, which in particular holds for $q=p$.}

\begin{claim}\label{claim:compatible1}
  $q\parallel\pi(q)$.
\end{claim}
\begin{proof}
  Let $r$ be the componentwise union $r=q\cup\pi(q)$, which makes sense as any $\gamma\in t(q)\cup t(\pi(q))$ is contained in exactly one of $t(q,e)$, $t(q)\setminus t(q,e)$ or $t(\pi(q))\setminus t(q,e)$ by our choice of $\pi$. In the first case, $q_\gamma=\pi(q)_\gamma$, while in the remaining two cases, $\gamma$ is contained in either $t(q)$ or $t(\pi(q))$, but not both simultaneously. We are left to show that $r$ has the HAD property and is thus a condition in $P$. The only nontrivial case is when $\gamma_0\in t(q)\setminus t(q,e)$ and $\gamma_1\in t(\pi(q))\setminus t(q,e)$. But then, the following hold:
  \begin{itemize}
      \item $t(r,\{\gamma_0\})=t(q,\{\gamma_0\})$.
      \item $\exists\,\gamma'\in t(q)\setminus t(q,e)\ \gamma_1=\pi(\gamma')$.
      \item $t(r,\{\gamma_1\})=t(\pi(q),\{\pi(\gamma')\})=\pi[t(q,\{\gamma'\})]$.
      \item By our choice of $\pi$, \[t(q,\{\gamma_0\})\cap\pi[t(q,\{\gamma'\})]\subseteq t(q,e),\] since already $t(q)\cap t(\pi(q))=t(q)\cap\pi[t(q)]\subseteq t(q,e)$.
  \end{itemize}
  We will be essentially done once we show the following:
  \begin{claim}
      $t(q,\{\gamma_0\})\cap\pi[t(q,\{\gamma'\})]=t(q,\{\gamma_0\})\cap t(q,\{\gamma'\})\cap t(q,e)$.
  \end{claim}
  \begin{proof}
    If $\bar\gamma$ is an element of the left hand side expression of the above equation, it follows that $\bar\gamma\in t(q,e)$ by the final of the above items. It thus follows that $\pi(\bar\gamma)=\bar\gamma$, which means that $\bar\gamma\in t(q,\{\gamma'\})$, and thus it is an element of the right hand side expression. In the other direction, if $\bar\gamma$ is an element of the right hand side expression, we again obtain that $\pi(\bar\gamma)=\bar\gamma$ and then that $\bar\gamma$ is an element of the left hand side expression.
  \end{proof}
  Now, since $q$ is HAD, using Lemma \ref{lemma:finitehad}, we find a finite $c\subseteq t(q)$ such that \[t(r,\{\gamma_0\})\cap t(r,\{\gamma_1\})=t(q,\{\gamma_0\})\cap t(q,\{\gamma'\})\cap t(q,e)=t(q,c).\]
  This finishes the argument to show that $r$ is a HAD tower.
\end{proof}
\end{proof}

\begin{theorem}
Let $G$ be $P$-generic.
  There is no choice function for the sequence $\langle A_\alpha\mid\alpha<\lambda\rangle$ in $V[G]_{\sS}$. This implies that $\DC_{\lambda}$, and hence in particular $\AC$ fails in $V[G]_{\sS}$.
\end{theorem}
\begin{proof}
Assume for a contradiction that $\dot F$ is an $\SSS$-name which is forced by some condition $p\in P$ to actually be such an choice function. Let $e\subseteq\lambda\times\lambda$ be finite such that $\fix(e)\le\sym(\dot F)$. Pick $\alpha<\lambda$ such that $\alpha>\max\dom(e)$. Pick $q\le p$ and $\beta<\lambda$ such that $q\forces\dot F(\check \alpha)=\dot g_{\alpha,\beta}$ and, using Lemma \ref{lemma:addtotarget}, $(\alpha, \beta) \in t(q)$. 
Pick a permutation $\pi=\langle\pi_\gamma\mid\gamma<\lambda\rangle\in\sG$ such that $\pi$ fixes $t(q,e)$ pointwise, and which swaps $t(q)\setminus t(q,e)$ with a set that is disjoint from $t(q)$. Such $\pi$ can easily be found, and since $(\alpha,\beta)\not\in t(q,e)$, $\pi(\alpha,\beta)=(\alpha,\beta')$ for some $\beta'\ne\beta$.
Then, $\forces\pi(\dot g_{\alpha,\beta})=\dot g_{\alpha,\beta'}\ne\dot g_{\alpha,\beta}$, and also $\pi(q)\forces\dot F(\check \alpha)=\dot g_{\alpha,\beta'}$. But this is a contradiction since $q\parallel\pi(q)$ by Claim \ref{claim:compatible1} -- note that we are in exactly the same situation as in that claim.
\end{proof}

\section{Minimal Supports}

In this section, we want to introduce a concept of minimal supports for $\sS$-names, and show that every $\sS$-name has such a minimal support.

\begin{definition}\label{definition:irreducible}
    Let $p\in P$. We say that a finite subset $a\subseteq t(p)$ is \emph{irreducible} (in $p$) if $t(p, b)\subsetneq t(p, a)$ whenever $b \subsetneq a$.
\end{definition}

\begin{lemma}\label{lemma:topintersection}
  If $\dot x$ and $\dot y$ are $\SSS$-names with finite supports $a,b\subseteq\lambda\times\lambda$ respectively, and $p\in P$ is such that $p\forces\dot x=\dot y$ and $a\cup b\subseteq t(p)$, then there is an irreducible $c\subseteq t(p,a)\cap t(p,b)$, and an $\SSS$-name $\dot z$ with $\fix(c)\le\sym(\dot z)$, such that $p\forces\dot z=\dot x$.
\end{lemma}
\begin{proof}
  Consider $$\dot y' = \{ (s, \tau) : \exists (r, \tau) \in \dot y\  s \leq r,p \}.$$
 Clearly, $p \Vdash \dot y = \dot y'$. Using the HAD property (together with the assumption that $a$ and $b$ are both finite), let $c\subseteq t(p,a)\cap t(p,b)$ be finite such that $t(p,c)=t(p,a)\cap t(p,b)$. By possibly shrinking $c$ by one element finitely many times, we may additionally assume that $c$ is irreducible.
 
Now, simply consider $$\dot z = \bigcup_{\pi\, \in\, \fix(c)} \pi(\dot y').$$
 We obviously have $\dot z\in\HS$ and $\fix(c) \leq \sym(\dot z)$. We claim that indeed $p \Vdash \dot z = \dot x$. Toward this end, let $G$ be an arbitrary $P$-generic containing the condition $p$. We already know that $\dot x^G =  (\dot y')^G = \id(\dot y')^G  \subseteq \dot z^G$. Thus, it suffices to show that for any $\pi \in \fix(c)$, $\pi(\dot y')^G \subseteq \dot x^G$. 
 
 So let $\pi\in\fix(c)$. If $\pi(p) \notin G$, clearly $\pi(\dot y')^G = \emptyset$, as every condition appearing in a pair in $\pi(\dot y')$ is below $\pi(p)$. So assume that $\pi(p) \in G$. Let \[d = \{ \gamma\in\lambda\times\lambda\mid \pi(\gamma) \neq \gamma\}\cup t(p),\] which is of size less than $\lambda$. Pick $\sigma=\langle\sigma_\alpha\mid \alpha<\lambda\rangle\in\fix(t(p,a))$ so that $\sigma$ swaps the elements of $b\setminus t(p,a)$ with pairs of ordinals in $(\lambda\times\lambda)\setminus d$, and such that $\sigma(p) \in G$. This is possible:
 
\begin{claim}
  For any $q\le p$ there exists $\sigma\in\fix(t(p,a))$ that swaps the elements of $b\setminus t(p,a)$ with pairs of ordinals in $(\lambda\times\lambda)\setminus d$, and for which we have $q\parallel\sigma(p)$. Thus, by the genericity of $G$, there exists a desired $\sigma$ with $\sigma(p)\in G$.
\end{claim}
\begin{proof}
  Let $q\le p$, and let $e=d\cup t(q)$. Pick $\sigma=\langle\sigma_\alpha\mid\alpha<\lambda\rangle\in\sG$ fixing $t(p,a)$ pointwise, and which swaps $t(q)\setminus t(p,a)$ with a set that is disjoint from~$e$. Such $\sigma$ can easily be found.  Remember that $t(q,a)=t(p,a)$. Arguing exactly as in Claim~\ref{claim:compatible1} (with $\sigma$ in place of $\pi$, and with $a$ in place of $e$),  we obtain the stronger conclusion that $q\parallel\sigma(q)$. Now, this shows that for any $q\le p$ there is a permutation~$\sigma$ which is as desired, and we may thus pick $r\le q,\sigma(p)$. This yields a dense set of conditions~$r$, so we may pick one such $r\in G$. For the corresponding permutation~$\sigma$, it thus follows that $\sigma(p)\in G$, as desired.
\end{proof}
 
 Then, note that $\sigma(p) \Vdash \dot x = \sigma(\dot y')=\sigma(\dot y)$.
 Note also that $\pi(\sigma(p) \cup p)\in G$, since, by the properties of $\sigma$, it is weaker than $\pi(p)\cup\sigma(p)\in G$.
 Since $\fix(b)\le\sym(\dot y)$, it follows that $\fix(\sigma[b])\le\sym(\sigma(\dot y))$. Let's take a closer look at $\sigma[b]$. It can be written as a disjoint union of $\mathfrak a:=\sigma[b]\cap t(p,a)$ and of $\mathfrak b:=\sigma[b]\setminus t(p,a)$.

The set $\mathfrak a$ is pointwise fixed by $\sigma$, because $t(p,a)$ is, so in fact, $\mathfrak a=b\cap t(p,a)\subseteq t(p,a)\cap t(p,b)\subseteq t(p,c)$. The set $\mathfrak b$ is pointwise fixed by $\pi$, as follows easily from the definition of $\sigma$. That is, $\pi\in\fix(c\cup\mathfrak b)$. We also have \[\sigma[b]=\mathfrak a\cup\mathfrak b\subseteq t(p\cup\sigma(p),c\cup\mathfrak b)\] by the above. Thus, by Lemma~\ref{lemma:supportup}, there is a name $\dot y^*\in\HS$ with $\fix(c\cup\mathfrak b)\le\sym(\dot y^*)$ and such that $p\cup\sigma(p) \forces\dot y^*=\sigma(\dot y)$. 
This means that $\pi\in\sym(\dot y^*)$, and therefore, $\pi(\sigma(p) \cup p) \cup p \forces\pi(\dot x)=\dot y^*=\sigma(\dot y)$.
 Overall, since also $\pi(p)\forces\pi(\dot x)=\pi(\dot y')$, it follows in particular that $\dot x^G=\sigma(\dot y)^G=\pi(\dot x)^G=\pi(\dot y')^G$, as desired.
\end{proof}

\begin{definition}
  Let $p\in P$. We define a relation $\unlhd_p$ on the set of all irreducible subsets of $t(p)$, letting, for $a, b$ irreducible in $p$, $a \unlhd_p b$ if $t(p, a) \subseteq t(p, b)$.

  We define the \emph{strict} relation $\lhd$ by setting $a \lhd b$ if $a \unlhd b \wedge a \neq b$.
\end{definition}

We will usually omit the subscript $p$ when the relevant tower is clear from context. Note also that if $q\le p$ are complete towers and $a\unlhd_p b$, then also $a\unlhd_q b$, and also if $a\unlhd_q b$ and $b\subseteq t(p)$, then also $a\subseteq t(p)$, and $a\unlhd_p b$.

\begin{lemma}\label{lemma:wellfounded}
    Let $p$ be a complete tower. Then, $\unlhd=\unlhd_p$ is a well-founded partial order.
\end{lemma}

\begin{proof}
    Clearly, $\unlhd$ is transitive and reflexive. In order to check antisymmetry, suppose for a contradiction that $t(p, a) = t(p, b)$ but $a \neq b$. Let $\alpha$ be largest so that $a_\alpha \neq b_\alpha$, where $a_\alpha := \{\beta\mid(\alpha, \beta) \in a \}$, and similarly for $b$. Say, without loss of generality, that $\beta \in a_\alpha \setminus b_\alpha$. As $(\alpha, \beta) \in t(p, a) = t(p,b)$, there must be some $\bar \alpha > \alpha$ and $\bar \beta < \lambda$ with $(\bar \alpha, \bar \beta) \in b$ and $(\alpha, \beta) \in t(p, \{(\bar \alpha, \bar \beta)\})$. But then $(\bar \alpha, \bar \beta) \in a$ as well, as $\alpha$ was chosen largest with $a_\alpha \neq b_\beta$. We obtain that $t(p, a) = t(p, a \setminus \{(\alpha, \beta)\})$, so $a$ is not irreducible, which is a contradiction.
    
    To check well-foundedness, for an irreducible $a \subseteq t(p)$, let $$\delta(a) := \sum_{\alpha \in \dom a} \omega^\alpha \cdot \vert a_\alpha \vert,$$ using ordinal arithmetic. It suffices to note that $a \lhd b$ implies $\delta(a) < \delta(b)$. Towards this end, again, let $\alpha$ be largest so that $a_\alpha \neq b_\alpha$. We claim that $a_\alpha \subseteq b_\alpha$. In particular then, $a_\alpha$ must be a strict subset of $b_\alpha$ and we obtain that $\delta(a) < \delta(b)$. So suppose otherwise, that there is $\beta \in a_\alpha \setminus b_\alpha$. Just as before, we obtain that $a$ is not irreducible, using that $t(p, a) \subseteq t(p, b)$, which is again a contradiction.
\end{proof}

\begin{theorem}[Minimal Supports]\label{theorem:minsup}
  If $\dot x$ is an $\SSS$-name and $p\in P$, then there is $q\le p$, a unique (with respect to $q$) irreducible (in $q$) $b\subseteq t(q)$, and $\dot y\in\HS$ with support $b$ for which $q\forces\dot y=\dot x$, and whenever $a\lhd b$ and $\dot z$ is an $\SSS$-name with support~$a$, then $q\forces\dot z\ne\dot x$. We say that $b$ is the \emph{minimal support for $\dot x$ below $q$} in this case.
\end{theorem}
\begin{proof}
  Use Lemma \ref{lemma:topintersection} repeatedly, in order to obtain successively stron\-ger conditions $q_i\le p$, $\sS$-names $\dot y_i$ and successively smaller (according to~$\lhd$) irreducible $b_i$, such that for each $i$, $q_i\forces\dot y_i=\dot x$ and $\fix(b_i)\le\sym(\dot y_i)$. By Lemma \ref{lemma:wellfounded}, this construction has to break down after a final finite stage $i$. Then clearly, $q_i$, $b_i$ and $\dot y_i$ are as desired, where the uniqueness of $b_i$ follows from the fact that $\unlhd$ is a partial order, that is if some irreducible $b$ satisfies $b\unlhd b_i$ and $b_i\unlhd b$, then already $b=b_i$.
\end{proof}

Note that if $b$ is the minimal support for an $\sS$-name $\dot x$ below a condition $q\in P$ and $r\le q$, then $b$ is also the minimal support for $\dot x$ below~$r$. Moreover, if $\pi\in\sG$, then $\pi[b]$ is the minimal support for $\pi(\dot x)$ below $\pi(q)$.

\section{The Ordering Principle}

We now want to show that the ordering principle holds in our symmetric extension. The arguments in this section will be very similar to the corresponding arguments presented in \cite{hs}.

\begin{lemma}\label{lem:Alin}
  There is an $\sS$-name $\dot <$ for a linear order of $\dot A$, such that $\sym(\dot <) = \sG$.
\end{lemma}
\begin{proof} In any model of $\ZF$, we can consider the definable sequence of sets $\langle X_\alpha : \alpha \in \Ord\rangle$, obtained recursively by setting $X_0 = {}^{\kappa}2$, $X_{\alpha + 1} = {}^\omega X_\alpha$ and $X_\alpha = \bigcup_{\beta < \alpha} X_\beta$ for limit $\alpha$. We can recursively define linear orders $<_\alpha$ on $X_\alpha$, by letting $<_0$ be the lexicographic ordering on ${}^{\kappa}2$, $<_{\alpha+1}$ be the lexicographic ordering on $X_{\alpha +1}$ obtained from $<_\alpha$, and for limit $\alpha$, $x <_\alpha y$ iff, for $\beta$ least such that $x \in X_\beta$, either $y \notin X_\gamma$ for all $\gamma \leq \beta$, or $y \in X_\beta$ and $x <_\beta y$. Then $<_{\lambda}$ is a definable linear order of $X_{\lambda}$. Note that $\dot A$ is forced to be contained in $X_{\lambda}$, and by Fact~\ref{fact:namebydef}, there is an $\sS$-name $\dot <$ as required. 
\end{proof}

\begin{theorem}\label{theorem:classfunction}
  There is a class $\sS$-name $\dot F$ for an injection of the symmetric extension by $\sS$ into $\Ord \times \dot A^{<\omega}$ such that $\sym(\dot F) = \sG$. In particular, $\OP$ holds in our symmetric extension.
\end{theorem}

  \begin{proof}
    Fix a global well-order $\prec$ of our ground model $V$, and let $G$ be $P$-generic over~$V$. We first provide a definition of such an injection $F$ in the full $P$-generic extension $V[G]$. Then, we will observe that all the parameters in this definition have symmetric names, which will let us directly build an $\sS$-name $\dot F$ for $F$.
    
    For each $a \in [\lambda\times\lambda]^{<\omega}$ and each enumeration $h = \langle \gamma_i : i < k \rangle$ of~$a$, define $\dot G_a = \{ \dot g_\gamma\mid\gamma\in a \}^\bullet$
    and $\dot t_h = \langle \dot g_{\gamma_i} : i < k \rangle^\bullet$. Define $\dot \Gamma = \{\pi(\dot G) : \pi \in \sG \}^\bullet$. While $\dot \Gamma$ is not an $\sS$-name in general, it is still a symmetric $P$-name. Let $\Gamma = \dot \Gamma^G$ and $< = \dot <^G$. Given $x \in V[G]_\sS$, $F(x)$ will be found as follows: 
    
     First, let $(p, \dot z, a, h)$ be $\prec$-minimal with the following properties: 
     
     \begin{enumerate}
         \item in $V$, $a$ is the minimal support for $\dot z$ below $p$,
         \item in $V$, $h$ is an enumeration of $a$ so that $p$ forces that $\dot t_h$ enumerates $\dot G_a$ in the order of $\dot <$,
         \item in $V[G]$, there is $H \in \Gamma$ with $p \in H$ and $\dot z^H = x$.
     \end{enumerate} Such a tuple certainly exists by Theorem~\ref{theorem:minsup} and since $G \in \Gamma$.
     \begin{claim}
     For any $H, K \in \Gamma$ with $p \in H, K$, the following are equivalent:
    \begin{enumerate}[label=(\alph*)]
        \item $(\dot t_h)^H = (\dot t_h)^K$,
        \item $\dot z^H = \dot z^K$.
    \end{enumerate}
    \end{claim}
    \begin{proof}
    Let $H, K \in \Gamma$, $p \in H, K$. $H$ is itself a $P$-generic filter, and $\dot \Gamma^H = \dot \Gamma^G = \Gamma$, as can be easily checked. Thus, there is $\pi \in \sG$ so that $K = \pi(\dot G)^H$. Now, note that $\pi(\dot G)^H = \pi^{-1}[H]$ and $(\dot t_h)^K = (\dot t_h)^{\pi^{-1}[H]} = \pi(\dot t_h)^H$. Similarly, $\dot z^K = \pi(\dot z)^H$.

    \medskip

    Suppose that $(\dot t_h)^H = (\dot t_h)^K$. Then, $(\dot t_h)^H = \pi(\dot t_h)^H$. By the way that permutations act on the names $\dot g_\gamma$ (see Section \ref{section:symmetricsystem}), and thus on $\dot t_h$, the only way this is possible is if $\pi(\gamma)=\gamma$ for every $\gamma\in a$. In other words, $\pi \in \fix(a)$. Thus, $\dot z^H = \pi(\dot z)^H = \dot z^K$.

    \medskip    

    Now, suppose that $\dot z^H = \dot z^K = \pi(\dot z)^H$. Since $p\in K=\pi^{-1}[H]$, it follows that $\pi(p)\in H$. Thus, there is $r\le p,\pi(p)$ in $H$ with $r\forces\dot z=\pi(\dot z)$. Since $a$ is the minimal support for $\dot z$ below $p$, and hence also below $r$, also $\pi[a]$ is the minimal support for $\pi(\dot z)$ below $\pi(p)$, hence also below $r$. But by the uniqueness property in Theorem~\ref{theorem:minsup}, this implies that $\pi[a]=a$.
    This also means that $\dot G_a = \pi(\dot G_a)$. As $p$ forces that $\dot t_h$ is the $\dot <$-enumeration of $\dot G_a$, $\pi(p)$ forces that $\pi( \dot t_h)$ is the $\pi(\dot <)$-enumeration of $\pi(\dot G_a)$. Since $p \in K$ and $\pi(p) \in H$, this implies that $(\dot t_h)^K = \pi( \dot t_h)^H$ is the enumeration of $\pi(\dot G_a)^H = \dot G_a^H$ according to $\pi(\dot <)^H =\,<$, which is exactly what $(\dot t_h)^H$ is.
    \end{proof}
      
By the claim, there is a unique $t \in A^{<\omega}$ so that $t = (\dot t_h)^H$, for some, or equivalently all, $H \in \Gamma$ with $p \in H$ and $\dot z^H = x$. We let $F(x) = (\xi, t)$, where $(p, \dot z, a, h)$ is the $\xi^\textrm{th}$ element of $V$ according to $\prec$. To see that this is an injection, assume that $x$ and $y$ both yield the same $(p, \dot z, a, h)$ and $t$. Let $H,K\in\Gamma$ with $p \in H, K$, and with $\dot z^H = x$, $\dot z^K = y$. By our definition, $ t = (\dot t_h)^H = (\dot t_h)^K$, and according to the claim, $x = \dot z^H = \dot z^K = y$.  This finishes the definition of $F$. 

The definition we have just given can be rephrased as $$\text{ $F(x) = y$ iff $\varphi(x, y, \Gamma, <)$},$$ where $\varphi$ is a first order formula using the parameters $\Gamma$ and $<$, and the only parameters that are not shown are parameters from $V$, such as the class $\prec$ or the class of tuples $(p, \dot z, a, h)$ so that (1) and (2) hold. Simply let $$\dot F = \{ (p, (\dot x, \dot y)^\bullet) : \dot x, \dot y \in \HS \wedge p \Vdash_P \varphi(\dot x, \dot y, \dot \Gamma, \dot <) \},$$ where the parameters from $V$ in $\varphi$ are replaced by their check-names. Then, $\dot F\subseteq P\times\HS$, and $\sym(\dot F) = \sG$, so $\dot F$ is a class $\sS$-name, as desired. 

\medskip
It follows that $\OP$ holds in any symmetric extension by $\sS$ since by Lemma \ref{lem:Alin} and Fact \ref{fact:namebydef}, there is an $\sS$-name for a linear order of $\Ord \times \dot A^{<\omega}$, which can be pulled back to produce a class that is a linear order of the sets of our symmetric extension using~$F$.
  \end{proof}

\section{Higher dependent choice}

Recall the symmetric $P$-name $\dot \Gamma = \{\pi(\dot G) : \pi \in \sG \}^\bullet$ from the previous proof. We need the following fairly general result:

\begin{lemma}\label{lemma:define}
    Let $\dot x$ be a $P$-name and $e\in [\lambda\times\lambda]^{<\omega}$ so that $\fix(e) \leq \sym(\dot x)$. Whenever $G$ is $P$-generic, $x = \dot x^G$ and $\Gamma = \dot \Gamma^G$, then $x$ is definable in $V[G]$ from elements of $V$, from $\Gamma$ and from $\langle g_\gamma\mid\gamma\in e\rangle$, as the only parameters.
\end{lemma}

\begin{proof}
    In $V[G]$, define $y$ to consist exactly of those $z$ so that $z \in \dot x^H$ for some $H \in \Gamma$ with $\dot g_\gamma^G = \dot g_\gamma^H$ for all $\gamma\in e$. We claim that $x = y$. Clearly, $x \subseteq y$ as $G \in \Gamma$. Now suppose that $H \in \Gamma$ is arbitrary, so that $\dot g_\gamma^G = \dot g_\gamma^H$ for all $\gamma\in e$. Then, $H = \pi(\dot G)^G$, for some $\pi \in \sG$. We obtain that $\dot g_\gamma^G = \dot g_\gamma^H = \pi(\dot g_\gamma)^G = \dot g_{\pi(\gamma)}^G$, for each $\gamma\in e$. But this is only possible if $\pi \in \fix(e)$. So also $\dot x^H = \pi(\dot x)^G = \dot x^G$, and we are done. 
\end{proof}

A key idea of our forcing construction is captured by the following lemma.

\begin{lemma}\label{lemma:supportup}
        Let $p \in P$ and $\dot y \in \HS$ have finite support $e_0 \subseteq t(p, e_1)$, for some $e_1 \in [t(p)]^{<\omega}$. Then, there is $\dot y^* \in \HS$ with support $e_1$ such that $p \Vdash \dot y = \dot y^*$.
\end{lemma}
\begin{proof}
  Using Lemma \ref{lemma:define}, whenever $G$ is $P$-generic, $y=\dot y^G$ and $\Gamma=\dot\Gamma^G$, then $y$ is definable (by a fixed formula that does not depend on the particular choice of generic $G$) in $V[G]$ from elements of $V$, from $\Gamma$ and from $\langle g_\gamma\mid\gamma\in e_0\rangle$ as the only parameters. Note that if $p \in G$, since $e_0 \subseteq t(p,e_1)$, each $g_\gamma$ for $\gamma \in e_0$ is definable in $V[G]$ from some $g_{\gamma'}$ with $\gamma'\in e_1$. More specifically, there is a finite sequence $n_0, \ldots, n_k$ of ordinals (in $V$, that can be read off from $p$) such that $p\forces \dot g_{\gamma'}(n_0)(n_1)\dots(n_k) = \dot g_\gamma$. So we can find a formula $\varphi$ such that $$p\forces\dot y=\{w\mid\varphi(w,\dot\Gamma,\langle \dot g_{\gamma'}\mid\gamma'\in e_1\rangle^\bullet,\check v)\}$$ for some $v\in V$.
  For some large enough $\xi$, define
  $$\dot y^*=\{(r,\dot w)\in P\times\HS_\xi\mid r\forces\varphi(\dot w,\dot\Gamma,\langle \dot g_{\gamma'}\mid\gamma'\in e_1\rangle^\bullet,\check v).$$
  We obtain that $\fix(e_1)\le\sym(\dot y^*)$ and $p\forces\dot y=\dot y^*$, as desired.
\end{proof}

\begin{theorem}
Let $G$ be $P$-generic. If $\lambda=\kappa$, then $V[G]_{\sS}$ is closed under ${<}\kappa$-sequences in $V[G]$. In particular thus, since $\DC_{<\kappa}$ holds in $V[G]\models\ZFC$, it follows that $\DC_{<\kappa}$ holds in $V[G]_{\sS}$.
\end{theorem}

\begin{proof}
Let $\vec x$ be a $\delta$-sequence of elements $\langle x_\epsilon\mid\epsilon<\delta\rangle$ of $V[G]_{\sS}$ in $V[G]$, for some cardinal $\delta<\kappa$. Let $x=\{x_\epsilon\mid\epsilon<\delta\}$ denote the range of $\vec x$, and let $\dot x$ and $\dot{\vec x}$ be $P$-names for $x$ and $\vec x$ respectively. For some $p\in G$ and some large enough ordinal~$\xi$, $p \Vdash \dot x \subseteq \HS_\xi^\bullet$. By further strengthening $p$, using that $P$ is ${<}\kappa$-closed, we can find a sequence of $\sS$-names $\langle\dot x_\epsilon\mid \epsilon<\delta\rangle$ so that $p \Vdash \dot{\vec x}$ is a function with domain $\delta$ and $\forall \epsilon<\delta\ \dot{\vec x}(\epsilon) = \dot x_\epsilon$. For each $\epsilon<\delta$, there is $e_\epsilon \in [\kappa\times \kappa]^{<\omega}$ so that $\fix(e_\epsilon) \leq \sym(\dot x_\epsilon)$. Let $\alpha<\kappa$ be a large enough ordinal so that for each $\epsilon<\delta$, there is such $e_\epsilon$ in $[\alpha\times \kappa]^{<\omega}$, and such that $\alpha\ge\dom(p)$. Let $e=\bigcup_{i<\delta}e_i$, which is of size at most $\delta<\kappa$. Using Lemma~\ref{lemma:addtotarget}, the ${<}\kappa$-closure of $P$, and Lemma~\ref{lemma:singletoncover}, let $q\le p$, $\alpha^*\ge\alpha$, and let $q \in G$ be a HAD tower with the property that $e \subseteq t(q)=t(q,\{(\alpha^*,0)\})$.

\medskip

 Fix some $\epsilon<\delta$. By Lemma \ref{lemma:supportup}, we find $\dot x_\epsilon'\in\HS$ with support $\{(\alpha^*,0)\}$ such that $q\forces\dot x_\epsilon=\dot x_\epsilon'$.
Let $\dot{\vec y}=\langle\dot x_\epsilon'\mid \epsilon<\delta\rangle^\bullet$. Then, $\fix(\{(\alpha^*,0)\})\le\sym(\dot{\vec y})$, and we obtain that $\vec x=\dot{\vec y\,}^G\in V[G]_{\sS}$, as desired.
\end{proof}

\begin{theorem}
    Let $G$ be $P$-generic. If $\lambda>\kappa^+$, then $\DC_{<\lambda}$ holds in $V[G]_{\sS}$.
\end{theorem}

\begin{proof}
     Suppose that $\dot T$ is an $\mathcal{S}$-name for a ${<}\delta$-closed (in the symmetric extension) tree without terminal nodes.
     Let $p^0=(p^0_0,\bar p^0) \in P$ be arbitrary. By possibly strengthening $p^0$, we may assume that the support of $\dot T$ is contained in $t(p^0)$. We want to find a condition $q\le p^0$ forcing that $\dot T$ contains an increasing sequence of length $\delta$ in order to verify the theorem. Without loss of generality, since $\lambda>\kappa^+$, we may assume that $\kappa < \delta < \lambda$, and that $\delta$ is regular. Fix a name $\dot F$ as obtained from Theorem~\ref{theorem:classfunction}. We will recursively define a decreasing sequence $\langle p^\xi : \xi < \delta \rangle$ in $P$, with each $p^\xi$ of the form $p^\xi=(p^0_0,\bar p^\xi)$, and a $\subseteq$-increasing sequence $\langle X_\xi : \xi < \delta \rangle$, where $X_\xi \subseteq \Ord$ and $\vert X_\xi \vert < \lambda$, for each $\xi < \delta$. Initially, we are given~$p^0$ and we let $X_0 = \emptyset$. At limit steps $\xi<\delta$, we let $X_{\xi} = \bigcup_{\xi' < \xi} X_{\xi'}$ and we pick $p^\xi$ to be a lower bound for $\langle p^{\xi'} : \xi' < \xi \rangle$. At successor steps, given $p= p^\xi$ and $X = X_\xi$, we proceed as follows. 

    First, by extending $p$, using Lemma \ref{lemma:singletoncover}, we can assume that there is $\gamma \in t(p)$ such that $t(p)=t(p, \{\gamma\})$. Fix, for now, a $P$-generic $G$ with $p \in G$, and let $T := \dot T^G$, $F := \dot F^G$ and $g_{\alpha, \beta} := \dot g_{\alpha, \beta}^G$, for every $(\alpha, \beta) \in \lambda \times \lambda$. Note that the least $\ZF$-model extending $V$ and containing $g_\gamma$ as an element is $$V(g_{\gamma}) = V[\langle g_{0,\beta}  : (0, \beta) \in t(p)\rangle],$$ which is an $\add(\kappa, t(p)\cap(\{0\}\times\lambda))$-generic extension, and thus a model of $\ZFC$. 
    Moreover define $A_p := \{ g_{\gamma'} : \gamma' \in t(p) \} \in V(g_\gamma)$ and note that $A_p$ has size $< \lambda$. In particular, $V(g_\gamma) \models \vert (X \times A_p^{<\omega})^{<\delta} \vert < \lambda$. Whenever $\langle (\eta_i, a_i) : i < \delta' \rangle \in (X \times A_p^{<\omega})^{<\delta} \cap V(g_\gamma)$, the sequence $\langle F^{-1}(\eta_i, a_i) : i < \delta' \rangle$ may or may not be a chain in~$T$. In case it is, since $T$ is closed under increasing sequences of length less than~$\delta$, there is some $(\eta, e) \in \Ord \times (\lambda \times \lambda)^{<\omega}$, so that $F^{-1}(\eta, g_e)$ is an upper bound, where~$g_e$ is defined as $\langle g_{e_i} : i < \vert e\vert \rangle$ when $e=\langle e_i\mid i<\vert e\vert\rangle$. All in all, in $V[G]$, there is $Y \subseteq \Ord$ and $E \subseteq \lambda \times \lambda$, both of size ${<}\lambda$, such that we can find pairs $(\eta,e)$ witnessing any of the above described instances within $Y \times E^{<\omega}$. Using the $\lambda$-cc of~$P_0$ (this uses that $\kappa^+ \leq \lambda$) and the ${<}\lambda$-closure of $P_1$, back in $V$, we can find $q \leq p$ of the form $q=(p^0_0,\bar q)$ and sets $Y$ and $E$ such that $q$ forces that $Y$, $E$ are as just described. Finishing our recursive definitions, let $X_{\xi +1} = X \cup Y$ and $p^{\xi +1 } \leq q$ such that $E\subseteq t(p^{\xi +1})$.
    Let $p$ be the greatest lower bound of $\langle p^\xi : \xi < \delta \rangle$, and let $X = \bigcup_{\xi < \delta } X_\xi$.  Using Lemma \ref{lemma:singletoncover}, let $q \leq p$ be such that $t(q) = t(q, \{\gamma\}) = t(p) \cup \{ \gamma \}$ for some $\gamma \in \lambda \times \lambda$. 
    
    Now suppose that $q \in G$, for a $P$-generic $G$. We let $T$, $F$ and $g_{e}$, for $e \in (\lambda \times \lambda)^{<\omega}$, be the evaluations by $G$ of the corresponding names just as before. Let $$\tilde T = \{ \left((\eta_0, e_0),(\eta_1, e_1)\right) \in (X \times t(p)^{<\omega})^2 : F^{-1}(\eta_0, g_{e_0}) <_T F^{-1}(\eta_1, g_{e_1}) \},$$ where $<_T$ is the order of $T$.
    Note, by Lemma~\ref{lem:restriction}, that $\tilde T \in V(g_\gamma)$, for all names used in its definition have supports that are contained in $t(p)$.

    \begin{claim}
        $\tilde T$ is ${<}\delta$-closed in $V(g_\gamma)$.
    \end{claim}

    \begin{proof}
        Let $\langle (\eta_i, e_i) : i < \delta' \rangle \in V(g_\gamma)$ be a decreasing sequence in $\tilde T$, for some $\delta' < \delta$. Remember that $V(g_\gamma) = V[\langle g_{0, \beta}: (0, \beta) \in t(p) \rangle]$, which is an $\add(\kappa, t(p)\cap(\{0\}\times\lambda))$-generic extension of $V$. Since $\add(\kappa, t(p)\cap(\{0\}\times\lambda))$ has the $\kappa^+$-cc and $\kappa < \delta $, there is $J \subseteq t(p)\cap(\{0\}\times\lambda)$ of size $ < \delta$ so that $\langle (\eta_i, e_i) : i < \delta' \rangle \in V[\langle g_{0, \beta}: \beta \in J \rangle]$. By the regularity of~$\delta$, there is $\xi < \delta$ such that $\eta_i \in X_\xi$ and $e_i \subseteq t(p_\xi)$ for each $i < \delta'$, and such that $0\times J \subseteq t(p_\xi)=t(p_\xi, \{\gamma'\})$ for some (unique) $\gamma' \in t(p_\xi)$. In particular then, $\langle (\eta_i, e_i) : i < \delta' \rangle \in V(g_{\gamma'})$, and we ensured in the next step of our above recursive construction that there is an upper bound in $V(g_\gamma)$.
    \end{proof}

    Finally, constructing a branch $\langle (\eta_i, e_i) : i < \delta \rangle$ through $\tilde T$ in $V(g_\gamma)\models\ZFC$, we find that $\langle F^{-1}(\eta_i, g_{e_i}) : i < \delta \rangle$ is a branch through $T$ in $V[G]_\sS$. 
\end{proof}

\bibliographystyle{plain}

\end{document}